\documentclass[12pt]{article}
\usepackage{amsmath,enumerate,amsfonts,amssymb,color,graphicx,amsthm}

\usepackage{hyperref}
\usepackage{todonotes}
\usepackage[normalem]{ulem}
\usepackage{slashed}

\usepackage{cite}

\numberwithin{equation}{section}

\def\RR{{\mathbb R}}

\newcounter{marnote}

\begin{document}
\newtheorem{thm}{Theorem}[section]
\newtheorem{Def}[thm]{Definition}
\newtheorem{lem}[thm]{Lemma}
\newtheorem{rem}[thm]{Remark}
\newtheorem{question}[thm]{Question}
\newtheorem{prop}[thm]{Proposition}
\newtheorem{cor}[thm]{Corollary}
\newtheorem{example}[thm]{Example}

\title{Mean oscillation gradient estimates for elliptic systems in divergence form with VMO coefficients}

\author{Luc Nguyen \footnote{Mathematical Institute and St Edmund Hall, University of Oxford, Andrew Wiles Building, Radcliffe Observatory Quarter, Woodstock Road, Oxford OX2 6GG, UK. Email: luc.nguyen@maths.ox.ac.uk.}}

\date{}

\maketitle

\centerline{\it Dedicated to Professor Duong Minh Duc on the occasion of his 70th birthday}

\begin{abstract}
We consider gradient estimates for $H^1$ solutions of linear elliptic systems in divergence form $\partial_\alpha(A_{ij}^{\alpha\beta} \partial_\beta u^j) = 0$. It is known that the Dini continuity of coefficient matrix $A = (A_{ij}^{\alpha\beta}) $ is essential for the differentiability of solutions. We prove the following results:
\begin{enumerate}[(a)]
\item If $A$ satisfies a condition slightly weaker than Dini continuity but stronger than belonging to VMO, namely that the $L^2$ mean oscillation $\omega_{A,2}$ of $A$ satisfies
\[
X_{A,2} := \limsup_{r\rightarrow 0} r \int_r^2 \frac{\omega_{A,2}(t)}{t^2} \exp\Big(C_* \int_{t}^R \frac{\omega_{A,2}(s)}{s}\,ds\Big)\,dt  < \infty,
\]
where $C_*$ is a positive constant depending only on the dimensions and the ellipticity, then $\nabla u \in BMO$.
\item If $X_{A,2} = 0$, then $\nabla u \in VMO$.
\item If $A \in VMO$ and if $\nabla u \in L^\infty$, then $\nabla u \in VMO$. 
\item Finally, examples satisfying $X_{A,2} = 0$ are given showing that it is not possible to prove the boundedness of $\nabla u$ in statement (b), nor the continuity of $\nabla u$ in statement (c).
\end{enumerate}
\end{abstract}

\section{Introduction}

Let $n \geq 2$, $N \geq 1$ and consider the elliptic system for $u = (u^1, \ldots, u^N)$
\begin{equation}
\partial_\alpha(A_{ij}^{\alpha\beta} \partial_\beta u^j) = 0 \quad \text{ in } B_4, \quad i = 1, \ldots, N,
	\label{Eq:u}
\end{equation}
where $B_4$ is the ball in $\RR^n$ of radius four and centered at the origin, and the coefficient matrix $A = (A_{ij}^{\alpha\beta})$ is assumed to be bounded and measurable in $\bar B_4$ and to satisfy, for some positive constants $\lambda$ and $\Lambda$,
\begin{align}
|A(x)|
	&\leq \Lambda \quad\text{ for a.e. } \quad x \in B_4,\label{Eq:CoefB}\\
\int_{B_2} A_{ij}^{\alpha\beta} \partial_\beta \varphi^j \partial_\alpha \varphi^i\,dx
 	& \geq \lambda \|\nabla \varphi\|_{L^2(B_4)}^2 \text{ for all } \varphi \in H_0^1(B_4).
 	\label{Eq:CoefEll}
\end{align}

It is well known that if the coefficient matrix $A$ belongs to $C^{0,\alpha}_{\rm loc}(B_4)$ then every solution $u \in H^1(B_4)$ of \eqref{Eq:u} belongs to $C^{1,\alpha}_{\rm loc}(B_2)$; see e.g. Giaquinta \cite[Theorem 3.2]{GiaquintaBook} where the result is attributed to Campanato \cite{Campanato65} and Morrey \cite{Morrey54}. It was conjectured by Serrin \cite{Serrin64} that the assumption $u \in H^1(B_4)$ can be relaxed to $u \in W^{1,1}(B_4)$. This has been settled in the affirmmative by Brezis \cite{BrezisinAncona09, Brezis08}. (See Hager and Ross \cite{HagerRoss72} for the relaxation from $u \in H^1(B_4)$ to $u \in W^{1,p}(B_4)$ for some $1 < p < 2$.) Moreover, in \cite{BrezisinAncona09, Brezis08}, it was shown that if $A$ satisfies the Dini condition
\begin{equation}
\int_0^2 \frac{\bar\omega_{A}(t)}{t}\,dt < \infty \quad \text{ where } \bar\omega_{A}(r) := \sup_{x, y \in B_2,|x - y| < r} |A(x) - A(y)|,
	\label{Eq:Dini}
\end{equation}
then every solution $u \in W^{1,1}(B_4)$ of \eqref{Eq:u} belongs to $C^1(B_2)$. For related works on the differentiability of weak solutions under suitable conditions on $\bar\omega_A$, see also \cite{HartmanWintner55, Lieberman87, MazyaMcOwen11-JDE}. 

Differentiability of weak solutions under weaker Dini conditions involving integral mean oscillation of $A$ has also been studied. For $0 < r \leq 2$, let
\begin{align*}
\bar\varphi_A(r) 
	&:= \sup_{x \in B_2}  \Big\{\frac{1}{|B_r(x)|}\int_{B_r(x)} |A(y) - A(x)|^2\,dy\Big\}^{1/2},\\
\omega_{A}(r) 
	&:= \sup_{x \in B_2} \frac{1}{|B_r(x)|}\int_{B_r(x)} |A(y) - (A)_{B_r(x)}|\,dy, \\
(A)_{B_r(x)} 
	&:= \frac{1}{|B_r(x)|}\int_{B_r(x)} A(y)\,dy, \quad 0 < r \leq 2.
\end{align*}
In Li \cite{Li17-CAM} it was shown that if 
\begin{equation}
\int_0^2 \frac{\bar\varphi_{A}(t)}{t}\,dt 
	< \infty,
	\label{Eq:DiniL2}
\end{equation}
then every solution $u \in H^1(B_4)$ of \eqref{Eq:u} belongs to $C^1(B_2)$. In Dong and Kim \cite{DongKim17-CPDE} (see also \cite{DongEscauriazaKim18-MAnn}), this conclusion was shown to remain valid under the weaker condition that 
\begin{align}
\int_0^2 \frac{\omega_{A}(t)}{t}\,dt 
	< \infty.
\label{Eq:DiniL1} 
\end{align}
(Note that the finiteness of $\int_0^2 \frac{\omega_{A}(t)}{t}\,dt$ or $\int_0^2 \frac{\bar\varphi_{A}(t)}{t}\,dt$ implies that $A$ is continuous.)

The Dini condition \eqref{Eq:Dini} and its integral variants \eqref{Eq:DiniL2}, \eqref{Eq:DiniL1} are phenomenologically sharp for the differentiablity of weak solutions of \eqref{Eq:u}. In Jin, Maz'ya and van Schaftingen \cite{JinMvSchaftingen09}, examples of continuous coefficient matrices $A$ with moduli of continuity $\bar\omega_A(t) \sim \frac{1}{|\ln t|}$ as $t \rightarrow 0$ were given showing the following phenomena: 
\begin{itemize}
\item there exists a solution $u \in W^{1,1}(B_4)$ of \eqref{Eq:u} such that  $u \in W^{1,p}(B_4)$ for all $p \in [1,\infty)$, and $\nabla u \in BMO_{\rm loc}(B_4)$ but $\nabla u \notin L^\infty_{\rm loc}(B_2)$ and $\nabla u \notin VMO_{\rm loc}(B_2)$.\footnote{The statement that $\nabla u \notin VMO_{\rm loc}(B_2)$ is not explicitly stated in \cite{JinMvSchaftingen09}, but can be seen from the proof of Proposition 1.5 therein.}
\item there exists a solution $u \in W^{1,1}(B_4)$ of \eqref{Eq:u} such that  $u \in W^{1,p}(B_4)$  for all $p \in [1,\infty)$ but $\nabla u \notin BMO_{\rm loc}(B_2)$.
\end{itemize}

In this paper, we consider mean oscillation estimates for $\nabla u$ when $A$ slightly fails the Dini conditions \eqref{Eq:Dini}, \eqref{Eq:DiniL2} and \eqref{Eq:DiniL1}. For $1 \leq p < \infty$, let $\omega_{A,p}: (0, 2] \rightarrow [0,\infty)$ denote the $L^p$ mean oscillation of $A$:
\begin{align*}
\omega_{A,p}(r) &= \sup_{x \in B_2} \Big\{\frac{1}{|B_r(x)|}\int_{B_r(x)} |A(y) - (A)_{B_r(x)}|^p\,dy\Big\}^{1/p}.
\end{align*}
It is clear that $\omega_{A,1} = \omega_{A}$, $\omega_{A,2} \leq \bar\varphi_A$, $\omega_{A,p}$ is non-decreasing in $p$, and $\omega_{A,p} \leq \bar\omega_A$ for all $p \in [1,\infty)$.

We now state our first result.

\begin{thm}
\label{Thm:1}
Let $A = (A_{ij}^{\alpha\beta})$ satisfy \eqref{Eq:CoefB} and \eqref{Eq:CoefEll}. There exists a constant $C_* > 0$, depending only on $n$, $N$, $\Lambda$ and $\lambda$ such that if
\begin{equation}
X_{A,2} := \limsup_{r\rightarrow 0} r \int_r^2 \frac{\omega_{A,2}(t)}{t^2} \exp\Big(C_* \int_{t}^2 \frac{\omega_{A,2}(s)}{s}\,ds\Big)\,dt < \infty,
	\label{Eq:DiniX}
\end{equation}
then every solution $u \in H^1(B_4)$ of \eqref{Eq:u} satisfies $\nabla u \in BMO_{\rm loc}(B_2)$. Moreover, if 
\begin{equation}
X_{A,2} = 0,
	\label{Eq:DiniY}
\end{equation}
then every solution $u \in H^1(B_4)$ of \eqref{Eq:u} satisfies $\nabla u \in VMO_{\rm loc}(B_2)$. 
\end{thm}

Note that condition \eqref{Eq:DiniX} implies that $\omega_{A,2}(t) \rightarrow 0$ as $t \rightarrow 0$ i.e. $A \in VMO_{\rm loc}(B_2)$.

\begin{rem}
Let $1 < p < \infty$. Theorem \ref{Thm:1} remains valid if $\omega_{A,2}$ is replaced by $\omega_{A,p}$ and the regularity assumption $u \in H^1(B_4)$ is replaced by $u \in W^{1,p}(B_4)$, where the constant $C_*$ is now allowed to depend also on $p$. For $p \geq 2$, this follows from the inequality $\omega_{A,2} \leq \omega_{A,p}$ for those $p$. For $1 < p < 2$, see Proposition \ref{Prop:MOp}.
\end{rem}

It is clear that if $\omega_{A,2}$ satisfies \eqref{Eq:DiniL2}, then it satisfies \eqref{Eq:DiniY} (and hence \eqref{Eq:DiniX}). The following lemma gives examples which satisfy \eqref{Eq:DiniY} but not necessarily \eqref{Eq:DiniL2}.

\begin{lem}\label{Lem:Ex}
If $\displaystyle\limsup_{t \rightarrow 0} \omega_{A,2}(t) \ln \frac{1}{t} < \frac{1}{C_*}$, then $X_{A,2} = 0$. If $\displaystyle\liminf_{t \rightarrow 0} \omega_{A,2}(t) \ln \frac{1}{t} > \frac{1}{C_*}$, then $X_{A,2} = \infty$.
\end{lem}

We note that, in case $\omega_{A,2}(t) \ln \frac{1}{t} \rightarrow 0$ as $t \rightarrow 0$, the BMO regularity of $\nabla u$ was proved by Acquistapace \cite{Acquistapace92-AMPA}. (See also \cite{Huang96-IUMJ}.)
 
By Lemma \ref{Lem:Ex}, an explicit example of $\omega_{A,2}$ satisfying \eqref{Eq:DiniY} (for any constant $C_*$) but not \eqref{Eq:DiniL2} is 
\[
\omega_{A,2}(t) \sim \frac{1}{\ln \frac{64}{t} (\ln \ln \frac{64}{t})^\beta}, \quad \beta \in (0,1].
\]
In addition, unlike \eqref{Eq:DiniL2} or \eqref{Eq:DiniL1}, \eqref{Eq:DiniY} does not imply that $A$ is continuous, e.g.
\[
A_{ij}^{\alpha\beta}(x) = \Big(2 + \sin \ln \ln \ln \frac{64}{|x|}\Big) \delta_{ij} \delta^{\alpha\beta}.
\]
(This can be checked using the fact that the function $s \mapsto \sin s$ is Lipschitz on $\RR$ and the fact that the function $x \mapsto \ell(x) := \ln \ln \ln \frac{64}{|x|}$ has $L^2$ mean oscillation $\omega_{\ell,2}(t) \sim \frac{1}{\ln \frac{64}{t} \ln \ln \frac{64}{t}}$.) 

When $A$ is merely of vanishing mean oscillation, we have the following result.

\begin{thm}
\label{Thm:2}
Let $A = (A_{ij}^{\alpha\beta})$ belong to $VMO(B_4)$ and satisfy \eqref{Eq:CoefB} and \eqref{Eq:CoefEll}. Then every solution $u \in W^{1,\infty}(B_4)$ of \eqref{Eq:u} satisfies $\nabla u \in VMO(B_2)$. 
\end{thm}

The obtained regularity in the above theorems appears sharp. As in \cite{JinMvSchaftingen09}, counterexamples can be produced to show that, under \eqref{Eq:DiniY}, 
\begin{itemize}
\item solutions of \eqref{Eq:u} may not have bounded gradients (though their gradients are of vanishing mean oscillation by Theorem \ref{Thm:1}),
\item $W^{1,\infty}$ solutions of \eqref{Eq:u} may not be differentiable (though their gradients are of vanishing mean oscillation by Theorem \ref{Thm:2}).
\end{itemize}

\begin{prop}\label{Prop:c1}
There exist a coefficient matrix $A = (A_{ij}^{\alpha\beta}) \in C(\bar B_4)$ satisfying \eqref{Eq:CoefB}, \eqref{Eq:CoefEll} and \eqref{Eq:DiniY} and a solution $u \in H^1(B_4)$ of \eqref{Eq:u} such that $\nabla u \in VMO(B_4)$ but $\nabla u \notin L^\infty_{\rm loc}(B_2)$.
\end{prop}

\begin{prop}\label{Prop:c2}
There exist a coefficient matrix $A = (A_{ij}^{\alpha\beta})\in C(\bar B_4)$ satisfying \eqref{Eq:CoefB}, \eqref{Eq:CoefEll} and \eqref{Eq:DiniY} and a solution $u \in H^1(B_4)$ of \eqref{Eq:u} such that $\nabla u \in L^\infty(B_4) \cap VMO(B_4)$ but $\nabla u \notin C(B_2)$.
\end{prop}

Theorem \ref{Thm:1} and Theorem \ref{Thm:2} are consequences of the following proposition on the mean oscillation of the gradient $\nabla u$ in terms of the $L^2$ mean oscillation $\omega_{A,2}$ of $A$.

\begin{prop}\label{Prop:MO}
Let $A= (A_{ij}^{\alpha\beta})$ satisfy \eqref{Eq:CoefB} and \eqref{Eq:CoefEll}. Then there exists a constant $C_* > 0$, depending only on $n$, $N$, $\Lambda$ and $\lambda$ such that for every $u \in H^1(B_4)$ satisfying \eqref{Eq:u} and for $0 < r \leq R/4 \leq 1/2$, there hold
\begin{equation}
\int_{B_r} |\nabla u|^2\,dx \leq \frac{C_*r^n}{R^n} \exp\Big(2C_* \int_{2r}^R \frac{\omega_{A,2}(t)}{t}\,dt\Big) \int_{B_R} |\nabla u|^2\,dx,
	\label{Eq:Est1}
\end{equation}
and
\begin{multline}
\int_{B_r} |\nabla u - (\nabla u)_r|^2\,dx \leq \frac{C_*r^{n+2}}{R^n} \int_{B_R} |\nabla u|^2\,dx \times\\
	\times \Big\{\int_{2r}^R \frac{\omega_{A,2}(t)}{t^2} \exp\Big(C_* \int_{t}^R \frac{\omega_{A,2}(s)}{s}\,ds\Big)\,dt\Big\}^2,
	\label{Eq:Est2}
\end{multline}
where $(\nabla u)_r = \frac{1}{|B_r|} \int_{B_r} \nabla u\,dx$ for $0 < r \leq 2$.
 
Moreover, if $u \in W^{1,\infty}(B_4)$, then, for $0 < r \leq R/4 \leq 1/2$,
\begin{equation}
\int_{B_r} |\nabla u - (\nabla u)_r|^2\,dx \leq \frac{C_*r^{n+2}}{R^n}\Big\{\int_{2r}^R \frac{\omega_{A,2}(t)}{t^2}\,dt\Big\}^2\,\sup_{B_R} |\nabla u|^2 .
	\label{Eq:Est3}
\end{equation}
\end{prop}

\begin{rem}
Let $1 < p < 2$. Under an additional assumption that $[A]_{BMO(B_4)}$ is sufficiently small, the estimates in Proposition \ref{Prop:MO} hold if $\omega_{A,2}$ is replaced by $\omega_{A,p}$ and the regularity assumption $u \in H^1(B_4)$ is replaced by $u \in W^{1,p}(B_4)$. We do know know if this smallness assumption can be dropped except for $p$ close to $2$. See Proposition \ref{Prop:MOp}.
\end{rem}

\bigskip
\noindent{\bf Acknowledgment.} The author would like to thank Professor Yanyan Li for drawing his attention to the problem.

\section{Proof of the main results}\label{Sec:Proof}

\begin{proof}[Proof of Lemma \ref{Lem:Ex}] We claim: For $\delta \in (0,1)$ and $a \in (0,\infty)$, the limit
\[
L_a = \limsup_{r \rightarrow 0} r \int_r^\delta \frac{1}{t^2} (\ln\frac{1}{t})^{a - 1}\,dt
\]
satisfies $L_a = \infty$ if $a > 1$, $L_a = 1$ if $a = 1$ and $L_a \leq (\ln \frac{1}{\delta})^{a-1}$ if $a < 1$.

When $a = 1$, the claim is clear. By integrating by parts, we have
\begin{equation}
\int_r^\delta \frac{1}{t^2} (\ln\frac{1}{t})^{a - 1}\,dt
	= - \frac{1}{t} (\ln\frac{1}{t})^{a - 1}\Big|_r^\delta - (a-1) \int_r^\delta \frac{1}{t^2} (\ln \frac{1}{t})^{a-2}\,dt.
	\label{Eq:22IV22-L1}
\end{equation}
If $a < 1$, we see from \eqref{Eq:22IV22-L1} that
\begin{align*}
L_a
	&= |a - 1| \limsup_{r \rightarrow 0} r\int_r^\delta \frac{1}{t^2} (\ln\frac{1}{t})^{a - 2}\,dt
		\leq |a - 1| \limsup_{r \rightarrow 0} \int_r^\delta \frac{1}{t} (\ln\frac{1}{t})^{a - 2}\,dt\\
	&=  \limsup_{r \rightarrow 0} (\ln\frac{1}{t})^{a - 1}\Big|_r^\delta
		= (\ln\frac{1}{\delta})^{a - 1}.
\end{align*}
To prove the claim in the case $a > 1$, we may assume without loss of generality that $a < 2$. Note that \eqref{Eq:22IV22-L1} implies
\[
L_a + (a-1) L_{a-1} = \limsup_{r \rightarrow 0 } r \Big\{- \frac{1}{t} (\ln\frac{1}{t})^{a - 1}\Big|_r^\delta\Big\} = \infty.
\]
As $L_{a-1}$ is finite (as $1 < a < 2$), we thus have that $L_a = \infty$. The claim is proved.

We now apply the claim to obtain the desired conclusions. Consider first the case that $\limsup_{t \rightarrow 0} \omega_{A,2}(t) \ln \frac{1}{t} < \frac{1}{C_*}$. Then there exist $\varepsilon \in (0,\frac{1}{C_*})$ and $\delta \in (0,1)$ so that $\omega_{A,2}(t) \leq \varepsilon (\ln \frac{1}{t})^{-1}$ in $(0,\delta)$. For $\hat \delta \in (0,\delta)$, we compute
\begin{align*}
X_{A,2} 
	&= \limsup_{r\rightarrow 0} r \int_r^{\hat\delta} \frac{\omega_{A,2}(t)}{t^2} \exp\Big(C_* \int_{t}^2 \frac{\omega_{A,2}(s)}{s}\,ds\Big)\,dt\\
	&\leq \varepsilon (\ln \frac{1}{\hat\delta})^{-C_*\varepsilon} \exp\Big(C_* \int_{\hat\delta}^2 \frac{\omega_{A,2}(s)}{s}\,ds\Big) \limsup_{r\rightarrow 0} r \int_r^{\hat\delta} \frac{1}{t^2}(\ln \frac{1}{t})^{C_* \varepsilon-1} \,dt.
\end{align*}
As $C_*\varepsilon < 1$, we can apply the claim to obtain
\begin{align*}
X_{A,2} 
	&\leq \varepsilon (\ln \frac{1}{\hat\delta})^{-1} \exp\Big(C_* \int_{\hat\delta}^2 \frac{\omega_{A,2}(s)}{s}\,ds\Big)\\ 
	&\leq \varepsilon (\ln \frac{1}{\hat\delta})^{-1 + C_* \varepsilon} (\ln \frac{1}{\delta})^{-C_*\varepsilon}  \exp\Big(C_* \int_{\delta}^2 \frac{\omega_{A,2}(s)}{s}\,ds\Big).
\end{align*}
Sending $\hat\delta \rightarrow 0$, we obtain that $X_{A,2} = 0$.

Consider next the case that $\liminf_{t \rightarrow 0} \omega_{A,2}(t) \ln \frac{1}{t} > \frac{1}{C_*}$. Then there exist $b  > \frac{1}{C_*}$ and $\delta \in (0,1)$ so that $\omega_{A,2}(t) \geq b (\ln \frac{1}{t})^{-1}$ in $(0,\delta)$. We then have
\begin{align*}
X_{A,2} 
	&= \limsup_{r\rightarrow 0} r \int_r^{\delta} \frac{\omega_{A,2}(t)}{t^2} \exp\Big(C_* \int_{t}^2 \frac{\omega_{A,2}(s)}{s}\,ds\Big)\,dt\\
	&\geq b (\ln \frac{1}{\delta})^{-C_*b} \exp\Big(C_* \int_{\delta}^2 \frac{\omega_{A,2}(s)}{s}\,ds\Big) \limsup_{r\rightarrow 0} r \int_r^{\delta} \frac{1}{t^2}(\ln \frac{1}{t})^{C_* b - 1} \,dt.
\end{align*}
As $C_*b > 1$, we deduce from the claim that $X_{A,2} = \infty$ as desired.
\end{proof}

\begin{proof}[Proof of Theorem \ref{Thm:1} and Theorem \ref{Thm:2}]
The results follow immediately from Proposition \ref{Prop:MO}.
\end{proof}

In order to prove Proposition \ref{Prop:MO}, we need the following estimate for harmonic replacements. (Compare  \cite[Lemma 3.5]{ByunWang04-CPAM}, \cite[Lemma 3.1]{LiNirenberg03-CPAM}.)

\begin{lem}\label{Lem:HRep}
Let $A, \bar A$ satisfy \eqref{Eq:CoefB} and \eqref{Eq:CoefEll} with $\bar A$ being constant in $B_4$ and $f = (f_i^\alpha) \in L^2(B_4)$. Let $R \in (0,2)$ and suppose $u, h \in H^1(B_{2R})$ satisfy
\begin{align*}
\partial_\alpha(A{}_{ij}^{\alpha\beta} \partial_\beta u^j) 
	&= \partial_\alpha f_i^\alpha\quad \text{ in } B_{2R}, \quad i = 1, \ldots, N,\\
\partial_\alpha(\bar A{}_{ij}^{\alpha\beta} \partial_\beta h^j) 
	&= 0 \quad \text{ in } B_{2R}, \quad i = 1, \ldots, N,\\
u
	& = h \quad  \text{ on } \partial B_{2R}.
\end{align*}
Then there exists a constant $C > 0$ depending only on $n, N, \Lambda$ and $\lambda$ such that
\[
\|\nabla (u-h)\|_{L^2(B_{3R/2})} \leq C \Big[\|f\|_{L^2(B_{2R})} + R^{-n/2} \|A - \bar A\|_{L^2(B_{2R})}  \|\nabla u\|_{L^2(B_{2R})}\Big].
\]
\end{lem}

\begin{proof} In the proof, $C$ denotes a generic positive constant which depends only on $n$, $N$, $\Lambda$ and $\lambda$. Using that $\bar A$ is constant, we have by standard elliptic estimates that
\[
\|\nabla h\|_{L^\infty(B_{7R/4})} \leq CR^{-n/2} \|\nabla h\|_{L^2(B_{2R})} \leq CR^{-n/2}\|\nabla u\|_{L^2(B_{2R})}.
\]
Observing that
\[
\partial_\alpha(A{}_{ij}^{\alpha\beta} \partial_\beta (u - h)^j) 
	= \partial_\alpha (f_i^\alpha + (\bar A - A)_{ij}^{\alpha\beta} \partial_\beta h^j) \quad \text{ in } B_{2R}, \quad i = 1, \ldots, N,
\]
we deduce that
\begin{align}
\|\nabla (u - h)\|_{L^2(B_{3R/2})} 
	&\leq C\Big[\|f\|_{L^2(B_{7R/4})} + \|A - \bar A\|_{L^2(B_{7R/4})}\|\nabla h\|_{L^\infty(B_{7R/4})}\nonumber\\
		&\qquad  + R^{-(n+2)/2}\|u - h\|_{L^1(B_{7R/4})}\Big]\nonumber\\
	&\leq C\Big[\|f\|_{L^2(B_{2R})} + R^{-n/2} \|A - \bar A\|_{L^2(B_{2R})}\|\nabla u\|_{L^2(B_{2R})} \nonumber\\
		&\qquad + R^{-(n+2)/2} \|u - h\|_{L^1(B_{2R})}\Big].
		\label{Eq:19IV22-E1}
\end{align}

To estimate $\|u - h\|_{L^1(B_{2R})}$, fix some $t > 0$ and consider an auxiliary equation
\begin{align*}
\partial_\beta(\bar A_{ij}^{\alpha\beta} \partial_\alpha \phi^i) 
	&= \frac{(u-h)^j}{\sqrt{|u-h|^2 + t^2}} \text{ in } B_{2R}, \quad j = 1, \ldots, N,\\
\phi
	& = 0 \quad  \text{ on } \partial B_{2R}.
\end{align*}
Testing the above against $u - h$, we obtain
\begin{equation}
\int_{B_{2R}} \frac{|u-h|^2}{\sqrt{|u-h|^2 + t^2}}\,dx = \int_{B_{2R}} \bar A_{ij}^{\alpha\beta} \partial_\alpha \phi^i \partial_\beta(u - h)^j\,dx.
	\label{Eq:19IV22-E2}
\end{equation}
As $u - h$ satisfies
\[
\partial_\alpha(\bar A{}_{ij}^{\alpha\beta} \partial_\beta (u - h)^j) 
	= \partial_\alpha (f_i^\alpha + (\bar A - A)_{ij}^{\alpha\beta} \partial_\beta u^j) \quad \text{ in } B_{2R}, \quad i = 1, \ldots, N,
\]
we have
\begin{equation}
\int_{B_{2R}} \bar A_{ij}^{\alpha\beta} \partial_\beta(u - h)^j\partial_\alpha \phi^i \,dx
	= \int_{B_{2R}} (f_i^\alpha + (\bar A - A)_{ij}^{\alpha\beta} \partial_\beta u^j)\partial_\alpha \phi^i \,dx.
	\label{Eq:19IV22-E2X}
\end{equation}
Inserting \eqref{Eq:19IV22-E2X} into \eqref{Eq:19IV22-E2} and noting that $\|\nabla\phi\|_{L^\infty(B_2)} \leq CR$ (as $|\partial_\beta(\bar A_{ij}^{\alpha\beta} \partial_\alpha \phi^i) | \leq 1$), we arrive at
\[
\int_{B_{2R}} \frac{|u-h|^2}{\sqrt{|u-h|^2 + t^2}}\,dx 
	\leq C\Big[R^{(n+2)/2} \|f\|_{L^2(B_{2R})} + R \|A - \bar A\|_{L^2(B_{2R})}\|\nabla u\|_{L^2(B_{2R})} \Big].
\]
Noting that the constant $C$ is independent of $t$, we may send $t \rightarrow 0$ to obtain
\begin{equation}
\|u-h\|_{L^1(B_{2R})}
	\leq CR^{(n+2)/2} \Big[\|f\|_{L^2(B_{2R})} + R^{-n/2}\|A - \bar A\|_{L^2(B_{2R})}\|\nabla u\|_{L^2(B_{2R})} \Big].
	\label{Eq:19IV22-E3}
\end{equation}
The conclusion follows from \eqref{Eq:19IV22-E1} and \eqref{Eq:19IV22-E3}.
\end{proof}

\begin{proof}[Proof of Proposition \ref{Prop:MO}]
We only need to give the proof for a fixed $R$, say $R = 2$. Our proof is inspired by that of \cite{Li17-CAM}.

In the proof, $C$ denotes a generic positive constant which depends only on $n$, $N$, $\Lambda$ and $\lambda$. In particular it is independent of the parameter $k$ which will appear below. Also, we will simply write $\omega$ instead of $\omega_{A,2}$.

\medskip
\noindent\underline{Proof of \eqref{Eq:Est1}:} For $k \geq 0$, let $R_k = 4^{-k}$, $\bar A_k = (A)_{B_{2R_k}}$ and $h_k \in H^1(B_{2R_k})$ be the solution to
\begin{align*}
\partial_\alpha((\bar A_k)_{ij}^{\alpha\beta} \partial_\beta h_{k}{}^j) 
	&= 0\quad \text{ in } B_{2R_k}, \quad i = 1, \ldots, N,\\
h_k
	&= u \quad \text{ on } \partial B_{2R_k}.
\end{align*}
Let $a_k = R_k^{-n/2}\|\nabla (u - h_k)\|_{L^2(B_{R_k})}$ and $b_k = \|\nabla h_k\|_{L^\infty(B_{R_k})}$. 

Note that, by triangle inequality, we have
\begin{equation}
\|\nabla u\|_{L^2(B_{R_k})} \leq R_k^{n/2}(a_k + b_k).
	\label{Eq:08IV22-X0}
\end{equation}

By elliptic estimates for $h_k$, we have
\begin{align}
\|\nabla h_k\|_{L^2(B_{2R_k})}
	& \leq C \|\nabla u\|_{L^2(B_{2R_k})},
	\label{Eq:20IV22-H1}\\
\|\nabla h_k\|_{L^\infty(B_{3R_k/2})} 
	&\leq C R_k^{-n/2}\|\nabla u\|_{L^2(B_{2R_k})},
	\label{Eq:20IV22-H2}\\
\|\nabla^2 h_k\|_{L^\infty(B_{3R_k/2})} + R_k \|\nabla^3 h_k\|_{L^\infty(B_{3R_k/2})}
	& \leq C R_k^{-(n+2)/2}\|\nabla u\|_{L^2(B_{2R_k})}.
	\label{Eq:20IV22-H3}
\end{align}

By Lemma \ref{Lem:HRep},
\begin{equation}
\|\nabla (u - h_k)\|_{L^2(B_{3R_k/2})} \leq C\omega(2R_k) \|\nabla u\|_{L^2(B_{2R_k})}.
	\label{Eq:08IV22-X1}
\end{equation}
By \eqref{Eq:08IV22-X1} and \eqref{Eq:20IV22-H1}, 
\[
R_k^{n/2}(a_k + b_k) \leq C\|\nabla u\|_{L^2(B_{2R_k})}.
\]

By \eqref{Eq:08IV22-X0} and \eqref{Eq:08IV22-X1}, we have 
\begin{align*}
\|\nabla (u - h_{k+1})\|_{L^2(B_{R_{k+1}})} 
	&\leq C\omega(2R_{k}) \|\nabla u\|_{L^2(B_{R_k})}  \nonumber	\\
	&\leq C\omega(2R_{k}) R_{k}^{n/2}(a_k + b_k).
\end{align*}
Hence
\begin{equation}
a_{k+1} \leq C\omega(2R_{k}) (a_k + b_k).
	\label{Eq:08IV22-X2}
\end{equation}

Next, we have by \eqref{Eq:08IV22-X1} that
\begin{align*}
\|\nabla (h_{k+1} - h_k)\|_{L^2(B_{3R_{k+1}/2})} 
	&\leq \|\nabla (u - h_{k+1})\|_{L^2(B_{3R_{k+1}/2})} + \|\nabla (u - h_k)\|_{L^2(B_{3R_{k+1}/2})} \\
	&\leq C\omega(2R_k) \|\nabla u\|_{L^2(B_{R_k})} \\
	&\leq C\omega(2R_{k}) R_{k}^{n/2}(a_k + b_k).
\end{align*}
Noting that $h_{k+1} - h_k$ satisfies
\begin{align*}
\partial_\alpha((\bar A_{k})_{ij}^{\alpha\beta} \partial_\beta (h_{k+1} - h_k)^j) 
	&= \partial_\alpha((\bar A_k - \bar A_{k+1})_{ij}^{\alpha\beta} \partial_\beta  h_{k+1}{}^j) \quad \text{ in } B_{2R_{k+1}}, \quad i = 1, \ldots, N,
\end{align*}
we thus have by elliptic estimates and \eqref{Eq:20IV22-H2} and \eqref{Eq:20IV22-H3} (applied to $h_{k+1}$) that
\begin{align}
\|\nabla (h_{k+1} - h_k)\|_{L^\infty(B_{R_{k+1}})} 
	&\leq C\omega(2R_{k}) (a_k + b_k),
	\label{Eq:08IV22-X3}\\
 R_{k+1}\|\nabla^2 (h_{k+1} - h_k)\|_{L^\infty(B_{R_{k+1}})}
	&\leq C\omega(2R_{k}) (a_k + b_k).
	\label{Eq:08IV22-X4}
\end{align}

By \eqref{Eq:08IV22-X3},
\begin{equation}
b_{k+1} \leq b_k + C\omega(2R_{k}) (a_k + b_k).
	\label{Eq:08IV22-X5}
\end{equation}
By \eqref{Eq:08IV22-X2} and \eqref{Eq:08IV22-X5}, we have
\[
a_{k+1} + b_{k+1} \leq (1 + C\omega(2R_{k})) (a_k + b_k).
\]
We deduce that
\begin{align}
a_k + b_k 
	&\leq \prod_{j=0}^k (1 + C\omega(2R_{j})) (a_0 + b_0) \leq C\exp\Big( C \sum_{j=0}^k\omega(2R_{j})\Big) \|\nabla u\|_{L^2(B_2)}\nonumber\\
	&\leq C\exp\Big( C \int_{2R_k}^2 \frac{\omega(t)}{t}\,dt\Big) \|\nabla u\|_{L^2(B_2)},
		\label{Eq:08IV22-X6}
\end{align}
where we have used the fact that $\omega(t) \leq C \omega(s)$ whenever $0 < t \leq s \leq 4t$. We have thus shown that 
\[
\int_{B_{R_k}} |\nabla u|^2\,dx \leq C R_k^n \exp\Big(C \int_{2R_k}^2 \frac{\omega(t)}{t}\,dt\Big) \int_{B_R} |\nabla u|^2\,dx \text{ for } k \geq 0.
\]
Estimate \eqref{Eq:Est1} is readily seen.

\medskip
\noindent\underline{Proof of \eqref{Eq:Est2}:} We write
\[
h_{R_{k}} = \sum_{j=0}^{k} w_j \text{ where } w_0 = h_{R_0} \text{ and } w_j = h_{R_{j}} - h_{R_{j-1}} \text{ for } j \geq 1.
\]
Using the estimate $\|\nabla^2 h_{R_0}\|_{L^\infty(B_1)} \leq C\|\nabla u\|_{L^2(B_2)}$ together with \eqref{Eq:08IV22-X4} and \eqref{Eq:08IV22-X6}, we have
\begin{align}
|\nabla h_{R_{k}}(x) - \nabla h_{R_{k}}(0)|
	&\leq C |x|\sum_{j=0}^{k} \frac{ \omega(2R_j) }{R_j} \exp\Big(C \int_{2R_j}^2 \frac{\omega(t)}{t}\,dt\Big) \|\nabla u\|_{L^2(B_2)}\nonumber\\
	&\leq C |x| \int_{2R_k}^2 \frac{ \omega(t)}{t^2}\, \exp\Big(C \int_{t}^2 \frac{\omega(s)}{s}\,ds\Big)\,dt \|\nabla u\|_{L^2(B_2)},
	\label{Eq:08IV22-X7P}
\end{align}
where we have again used the fact that $\omega(t) \leq C \omega(s)$ whenever $0 < t \leq s \leq 4t$. This implies
\begin{multline}
\|\nabla h_{R_{k}} - \nabla h_{R_{k}}(0)\|_{L^2(B_{R_k})}\\
	 \leq C R_k^{(n+2)/2}  \int_{2R_k}^2 \frac{ \omega(t)}{t^2}\, \exp\Big(C \int_{t}^2 \frac{\omega(s)}{s}\,ds\Big)\,dt  \|\nabla u\|_{L^2(B_2)}.
	\label{Eq:08IV22-X7}
\end{multline}
Combining \eqref{Eq:08IV22-X7} with \eqref{Eq:08IV22-X1} and \eqref{Eq:08IV22-X6}, we get
\begin{align}
\|\nabla u - (\nabla u)_{R_{k}}\|_{L^2(B_{R_k})}
	&\leq \|\nabla u - \nabla h_{R_{k}}(0)\|_{L^2(B_{R_k})}\nonumber\\
	&\leq \|\nabla (u - \nabla h_{R_{k}})\|_{L^2(B_{R_k})}
		+ \|\nabla u - \nabla h_{R_{k}}(0)\|_{L^2(B_{R_k})}\nonumber\\
	&\leq C R_k^{(n+2)/2} \int_{2R_k}^2 \frac{ \omega(t)}{t^2}\, \exp\Big(C \int_{t}^2 \frac{\omega(s)}{s}\,ds\Big)\,dt  \|\nabla u\|_{L^2(B_2)}\nonumber\\
		&\quad + C R_k^{n/2}  \omega(2R_k) \exp\Big(C \int_{2R_k}^2 \frac{\omega(t)}{t}\,dt\Big)  \|\nabla u\|_{L^2(B_2)}.
		\label{Eq:08IV22-X8}
\end{align}
As $\omega(2R_k) \leq C \omega(t)$ whenever $2R_k \leq t \leq 4R_k$, we have
\[
 \int_{2R_k}^{4R_k} \frac{ \omega(t)}{t^2}\, \exp\Big(C \int_{t}^2 \frac{\omega(s)}{s}\,ds\Big)\,dt 
 	 \geq \frac{\omega(2R_k)}{CR_k} \exp\Big(C \int_{2R_k}^2 \frac{\omega(s)}{s}\,ds\Big)
\]
Using this in \eqref{Eq:08IV22-X8}, we deduce that for $k \geq 1$ that
\begin{align*}
\|\nabla u - (\nabla u)_{R_{k}}\|_{L^2(B_{R_k})}
	&\leq C R_k^{(n+2)/2} \int_{2R_k}^2 \frac{ \omega(t)}{t^2}\, \exp\Big(C \int_{t}^2 \frac{\omega(s)}{s}\,ds\Big)\,dt  \|\nabla u\|_{L^2(B_2)}.
\end{align*}
Estimate \eqref{Eq:Est2} follows.

\medskip
\noindent\underline{Proof of \eqref{Eq:Est3}:} We adjust the proof of \eqref{Eq:Est2} exploiting the fact that $\nabla u \in L^\infty(B_2)$. First, using the fact that $a_k + b_k \leq CR_k^{n/2} \|\nabla u\|_{L^\infty(B_2)}$ in \eqref{Eq:08IV22-X4} we get instead of \eqref{Eq:08IV22-X7P} the stronger estimate
\begin{align}
|\nabla h_{R_{k}}(x) - \nabla h_{R_{k}}(0)|
	&\leq C |x| \int_{2R_k}^2 \frac{ \omega(t)}{t^2}\, dt \|\nabla u\|_{L^\infty(B_2)},
	\label{Eq:08IV22-X7PN}
\end{align}
and so
\begin{equation}
\|\nabla h_{R_{k}} - \nabla h_{R_{k}}(0)\|_{L^2(B_{R_k})}\\
	 \leq C R_k^{(n+2)/2}  \int_{2R_k}^2 \frac{ \omega(t)}{t^2}\, dt  \|\nabla u\|_{L^\infty(B_2)}.
	\label{Eq:08IV22-X7N}
\end{equation}
Combining \eqref{Eq:08IV22-X7N} with \eqref{Eq:08IV22-X1}, we get for $k \geq 1$ that
\begin{align}
\|\nabla u - (\nabla u)_{R_{k}}\|_{L^2(B_{R_k})}
	&\leq \|\nabla u - \nabla h_{R_{k}}(0)\|_{L^2(B_{R_k})}\nonumber\\
	&\leq \|\nabla (u - \nabla h_{R_{k}})\|_{L^2(B_{R_k})}
		+ \|\nabla u - \nabla h_{R_{k}}(0)\|_{L^2(B_{R_k})}\nonumber\\
	&\leq C R_k^{(n+2)/2} \int_{2R_k}^2 \frac{ \omega(t)}{t^2}\,dt  \|\nabla u\|_{L^\infty(B_2)}\nonumber\\
		&\quad + C R_k^{n/2}  \omega(2R_k)  \|\nabla u\|_{L^\infty(B_2)}\nonumber\\
	&\leq C R_k^{(n+2)/2} \int_{2R_k}^2 \frac{ \omega(t)}{t^2}\,dt  \|\nabla u\|_{L^\infty(B_2)}.
		\label{Eq:08IV22-X8N}
\end{align}
Estimate \eqref{Eq:Est3} follows.
\end{proof}

\begin{rem}
If the Dini condition \eqref{Eq:Dini} or \eqref{Eq:DiniL2} holds, it can be seen from \eqref{Eq:08IV22-X3} that $\{\nabla h_k(0)\}$ converges to some $P \in \RR^{N \times n}$, from which it follows that
\[
\lim_{r\rightarrow 0} r^{-n/2}\|\nabla u - P \|_{L^2(B_{r})} = 0,
\]
yielding the continuity of $\nabla u$ at the origin. We have thus recovered the results on the differentiability of $H^1$ solutions of Brezis \cite{BrezisinAncona09, Brezis08} and Li \cite{Li17-CAM}.
\end{rem}

\begin{proof}[Proof of Proposition \ref{Prop:c1}]
We take $N = 1$ and drop the indices $i$, $j$ in the expression of $A$ (so that $A= (A^{\alpha\beta})$). Following \cite[Lemma 2.1]{JinMvSchaftingen09}, we make the ansatz that
\begin{align*}
A^{\alpha\beta}(x)
	&= \delta^{\alpha\beta} + a(|x|) \Big(\delta^{\alpha\beta} - \frac{x^\alpha x^\beta}{|x|^2}\Big),\\
u(x)
	&= x^1 v(|x|).
\end{align*}
Then
\begin{align*}
\partial_\alpha (A^{\alpha\beta}\partial_\beta u) = x^1 \Big(v''(|x|) + \frac{n+1}{|x|} v'(|x|) - \frac{n-1}{|x|^2} a(|x|) v(|x|)\Big).
\end{align*}
Selecting now
\begin{align*}
a(r) 
	&= - \frac{1 + n \ln \frac{64}{r}}{(n-1) (\ln \frac{64}{r})^2 \ln \ln \frac{64}{r}},\\
v(r)
	&= \ln \ln \frac{64}{r},
\end{align*}
we see that $A$ is continuous in $\bar B_4$, satisfies \eqref{Eq:CoefB}, \eqref{Eq:CoefEll} and $u$ is an $H^1$ solution of \eqref{Eq:u}. The matrix $A$ admits a modulus of continuity $\bar\omega_A(t) \sim \frac{1}{\ln \frac{64}{t} \ln \ln \frac{64}{t}}$ as $t \rightarrow 0$ and so \eqref{Eq:DiniY} holds. It is readily seen that $u \in W^{1,p}(B_4)$ for all $p \in [1,\infty)$, $\nabla u \in VMO(B_4)$ but $\nabla u \notin L^\infty_{\rm loc}(B_2)$.
\end{proof}

\begin{proof}[Proof of Proposition \ref{Prop:c2}]
Instead of the choice in the proof of Proposition \ref{Prop:c1}, we now choose
\begin{align*}
a(r) 
	&= - \frac{\sin \ln \ln \ln \frac{64}{r} + \cos \ln \ln \ln \frac{64}{r}(1+ \ln \frac{64}{r} + n \ln \frac{64}{r} \ln \ln \frac{64}{r}) }{(n-1) (\ln \frac{64}{r})^2 (\ln \ln \frac{64}{r})^2 (2 + \sin\ln \ln \frac{64}{r})},\\
v(r)
	&= 2 + \sin\ln \ln \ln \frac{64}{r}.
\end{align*}
It is readily checked that $A$ is continuous in $\bar B_4$, satisfies \eqref{Eq:CoefB}, \eqref{Eq:CoefEll}, \eqref{Eq:DiniY} and $u$ is an $H^1$ solution of \eqref{Eq:u}, $\nabla u  \in L^\infty(B_4) \cap VMO(B_4)$ but $\nabla u \notin C(B_2)$.
\end{proof}

Finally, we briefly touch on the validity of Theorem \ref{Thm:1} when $\omega_{A,2}$ is replaced by $\omega_{A,p}$ for $1 < p < 2$. For this, we only need the following $L^p$ version of Proposition \ref{Prop:MO}.

\begin{prop}\label{Prop:MOp}
Let $A= (A_{ij}^{\alpha\beta})$ satisfy \eqref{Eq:CoefB} and \eqref{Eq:CoefEll}. Let $1 < p < 2$. Then there exist constants $\gamma  >0$ and $C_* > 0$ depending only on $n$, $N$, $p$, $\Lambda$ and $\lambda$ such that, provided $[A]_{BMO}(B_4) < \gamma$, there hold for every $u \in W^{1,p}(B_4)$ satisfying \eqref{Eq:u} and for $0 < r \leq R/4 \leq 1/2$ that
\begin{equation}
\int_{B_r} |\nabla u|^{p}\,dx \leq \frac{C_*r^n}{R^n} \exp\Big(2C_* \int_{2r}^R \frac{\omega_{A,p}(t)}{t}\,dt\Big) \int_{B_R} |\nabla u|^{p}\,dx,
	\label{Eq:Est1p}
\end{equation}
and
\begin{multline}
\int_{B_r} |\nabla u - (\nabla u)_r|^{p}\,dx \leq \frac{C_*r^{n+2}}{R^n} \int_{B_R} |\nabla u|^{p}\,dx \times\\
	\times \Big\{\int_{2r}^R \frac{\omega_{A,p}(t)}{t^2} \exp\Big(C_* \int_{t}^R \frac{\omega_{A,p}(s)}{s}\,ds\Big)\,dt\Big\}^2,
	\label{Eq:Est2p}
\end{multline}
where $(\nabla u)_r = \frac{1}{|B_r|} \int_{B_r} \nabla u\,dx$ for $0 < r \leq 2$.
 
Moreover, if $u \in W^{1,\infty}(B_4)$, then, for $0 < r \leq R/4 \leq 1/2$,
\begin{equation}
\int_{B_r} |\nabla u - (\nabla u)_r|^{p}\,dx \leq \frac{C_*r^{n+2}}{R^n}\Big\{\int_{2r}^R \frac{\omega_{A,p}(t)}{t^2}\,dt\Big\}^2\,\sup_{B_R} |\nabla u|^{p} .
	\label{Eq:Est3p}
\end{equation}
\end{prop}

The proof of Proposition \ref{Prop:MOp} is the same as that of Proposition \ref{Prop:MO}, but now using the following harmonic replacement estimate:

\begin{lem}\label{Lem:HRepp}
Let $1 < p < 2$. Let $A, \bar A$ satisfy \eqref{Eq:CoefB} and \eqref{Eq:CoefEll} with $\bar A$ being constant in $B_4$ and $f = (f_i^\alpha) \in L^{p'}(B_4)$. Let $R \in (0,1)$ and suppose $u, h \in W^{1,p}(B_{4R})$ satisfy
\begin{align*}
\partial_\alpha(A{}_{ij}^{\alpha\beta} \partial_\beta u^j) 
	&= \partial_\alpha f_i^\alpha\quad \text{ in } B_{3R}, \quad i = 1, \ldots, N,\\
\partial_\alpha(\bar A{}_{ij}^{\alpha\beta} \partial_\beta h^j) 
	&= 0 \quad \text{ in } B_{2R}, \quad i = 1, \ldots, N,\\
u
	& = h \quad  \text{ on } \partial B_{2R}.
\end{align*}
Then there exist constants $\gamma > 0$ and $C > 0$ depending only on $n, N, p, \Lambda$ and $\lambda$ such that, provided $[A]_{BMO(B_{4R})} \leq \gamma$,
\[
\|\nabla (u-h)\|_{L^{p}(B_{3R/2})} \leq C \Big[R^{n(1/p - 1/p')}\|f\|_{L^{p'}(B_{3R})} + R^{-n/p} \|A - \bar A\|_{L^p(B_{3R})}  \|\nabla u\|_{L^{p}(B_{3R})}\Big].
\]
\end{lem}

\begin{proof} We amend the proof of Lemma \ref{Lem:HRep} using $L^p$ theories for elliptic systems whose leading coefficients have small $BMO$ semi-norm.\footnote{When $p$ is close to $2$ such smallness assumption is not needed, see e.g. \cite{CaffarelliPeral98-CPAM, Stroffolini01-PA}.} In the proof, $C$ denotes a generic positive constant which depends only on $n$, $N$, $p$, $\Lambda$ and $\lambda$. 

It is known that (see e.g. Dong and Kim \cite{DongKim09-MAA, DongKim10-ARMA})\footnote{For further references, see \cite{Byun05-TrAMS, ByunWang04-CPAM, Fazio96, Krylov07-CPDE, Stroffolini01-PA}.}, provided $[A]_{BMO(B_{4R})} \leq \gamma$ for some small enough $\gamma$ depending only on $n, N, p, \Lambda$ and $\lambda$, one has
\begin{equation}
\|\nabla u\|_{L^{p'}(B_{2R})} \leq C\Big[\|f\|_{L^{p'}(B_{3R})} + R^{n(1/p' - 1/p)}\|\nabla u\|_{L^p(B_{3R})}\Big].
	\label{Eq:19IV22-E0p}
\end{equation}

Using that $\bar A$ is constant, we have by standard elliptic estimates that
\[
\|\nabla h\|_{L^\infty(B_{7R/4})} \leq CR^{-n/{p}} \|\nabla h\|_{L^{p}(B_{2R})} \leq CR^{-n/{p}}\|\nabla u\|_{L^{p}(B_{2R})}.
\]
Using
\[
\partial_\alpha(A{}_{ij}^{\alpha\beta} \partial_\beta (u - h)^j) 
	= \partial_\alpha (f_i^\alpha + (\bar A - A)_{ij}^{\alpha\beta} \partial_\beta h^j) \quad \text{ in } B_{2R}, \quad i = 1, \ldots, N,
\]
and once again the fact that $[A]_{BMO(B_{4R})} \leq \gamma$, we have
\begin{align}
\|\nabla (u - h)\|_{L^{p}(B_{3R/2})} 
	&\leq C\Big[\|f\|_{L^{p}(B_{7R/4})} + \|A - \bar A\|_{L^{p}(B_{7R/4})}\|\nabla h\|_{L^\infty(B_{7R/4})}\nonumber\\
		&\qquad  + R^{-(n+p')/{p'}}\|u - h\|_{L^1(B_{7R/4})}\Big]\nonumber\\
	&\leq C\Big[R^{n(1/p - 1/p')}\|f\|_{L^{p'}(B_{2R})} + R^{-n/{p}} \|A - \bar A\|_{L^{p}(B_{2R})}\|\nabla u\|_{L^{p}(B_{2R})} \nonumber\\
		&\qquad + R^{-(n+{p'})/{p'}} \|u - h\|_{L^1(B_{2R})}\Big].
		\label{Eq:19IV22-E1p}
\end{align}

To estimate $\|u - h\|_{L^1(B_{2R})}$, recall from the proof of Lemma \ref{Lem:HRep} the chain of identities
\begin{align*}
\int_{B_{2R}} \frac{|u-h|^2}{\sqrt{|u-h|^2 + t^2}}\,dx 
	&= \int_{B_{2R}} \bar A_{ij}^{\alpha\beta} \partial_\alpha \phi^i \partial_\beta(u - h)^j\,dx\\
	&= \int_{B_{2R}} (f_i^\alpha + (\bar A - A)_{ij}^{\alpha\beta} \partial_\beta u^j)\partial_\alpha \phi^i \,dx,
\end{align*}
which imply
\[
\int_{B_{2R}} \frac{|u-h|^2}{\sqrt{|u-h|^2 + t^2}}\,dx 
	\leq C\Big[R^{(n+p)/{p}} \|f\|_{L^{p'}(B_{2R})} + R \|A - \bar A\|_{L^p(B_{2R})}\|\nabla u\|_{L^{p'}(B_{2R})} \Big].
\]
Noting that the constant $C$ is independent of $t$, we may send $t \rightarrow 0$ to obtain
\begin{equation}
\|u-h\|_{L^1(B_{2R})}
	\leq CR^{(n+p)/p} \Big[\|f\|_{L^{p'}(B_{2R})} + R^{-n/p}\|A - \bar A\|_{L^p(B_{2R})}\|\nabla u\|_{L^{p'}(B_{2R})} \Big].
	\label{Eq:19IV22-E3p}
\end{equation}
The conclusion follows from \eqref{Eq:19IV22-E0p}, \eqref{Eq:19IV22-E1p} and \eqref{Eq:19IV22-E3p}.
\end{proof}

\newcommand{\noopsort}[1]{}


\begin{thebibliography}{10}

\bibitem{Acquistapace92-AMPA}
{\sc P.~Acquistapace}, {\em On {BMO} regularity for linear elliptic systems},
  Ann. Mat. Pura Appl. (4), 161 (1992), pp.~231--269.

\bibitem{Brezis08}
{\sc H.~Brezis}, {\em On a conjecture of {J}. {S}errin}, Atti Accad. Naz.
  Lincei Rend. Lincei Mat. Appl., 19 (2008), pp.~335--338.

\bibitem{BrezisinAncona09}
{\sc H.~Brezis}, {\em Solution of a conjecture by {J}. {S}errin {\rm in}
  \textsc{A. Ancona}, {E}lliptic operators, conormal derivatives and positive
  parts of functions}, J. Funct. Anal., 257 (2009), pp.~2124--2158.

\bibitem{Byun05-TrAMS}
{\sc S.-S. Byun}, {\em Elliptic equations with {BMO} coefficients in
  {L}ipschitz domains}, Trans. Amer. Math. Soc., 357 (2005), pp.~1025--1046.

\bibitem{ByunWang04-CPAM}
{\sc S.-S. Byun and L.~Wang}, {\em Elliptic equations with {BMO} coefficients
  in {R}eifenberg domains}, Comm. Pure Appl. Math., 57 (2004), pp.~1283--1310.

\bibitem{CaffarelliPeral98-CPAM}
{\sc L.~A. Caffarelli and I.~Peral}, {\em On {$W^{1,p}$} estimates for elliptic
  equations in divergence form}, Comm. Pure Appl. Math., 51 (1998), pp.~1--21.

\bibitem{Campanato65}
{\sc S.~Campanato}, {\em Equazioni ellittiche del {${\rm II}^\circ$} ordine e
  spazi {${\mathfrak{L}}^{(2,\lambda)}$}}, Ann. Mat. Pura Appl. (4), 69 (1965),
  pp.~321--381.

\bibitem{Fazio96}
{\sc G.~Di~Fazio}, {\em {$L^p$} estimates for divergence form elliptic
  equations with discontinuous coefficients}, Boll. Un. Mat. Ital. A (7), 10
  (1996), pp.~409--420.

\bibitem{DongEscauriazaKim18-MAnn}
{\sc H.~Dong, L.~Escauriaza, and S.~Kim}, {\em On {$C^1$}, {$C^2$}, and weak
  type-{$(1,1)$} estimates for linear elliptic operators: part {II}}, Math.
  Ann., 370 (2018), pp.~447--489.

\bibitem{DongKim09-MAA}
{\sc H.~Dong and D.~Kim}, {\em Parabolic and elliptic systems with {VMO}
  coefficients}, Methods Appl. Anal., 16 (2009), pp.~365--388.

\bibitem{DongKim10-ARMA}
\leavevmode\vrule height 2pt depth -1.6pt width 23pt, {\em Elliptic equations
  in divergence form with partially {BMO} coefficients}, Arch. Ration. Mech.
  Anal., 196 (2010), pp.~25--70.

\bibitem{DongKim17-CPDE}
{\sc H.~Dong and S.~Kim}, {\em On {$C^1$}, {$C^2$}, and weak type-{$(1,1)$}
  estimates for linear elliptic operators}, Comm. Partial Differential
  Equations, 42 (2017), pp.~417--435.

\bibitem{GiaquintaBook}
{\sc M.~Giaquinta}, {\em Multiple integrals in the calculus of variations and
  nonlinear elliptic systems}, vol.~105 of Annals of Mathematics Studies,
  Princeton University Press, Princeton, NJ, 1983.

\bibitem{HagerRoss72}
{\sc R.~A. Hager and J.~Ross}, {\em A regularity theorem for linear second
  order elliptic divergence equations}, Ann. Scuola Norm. Sup. Pisa Cl. Sci.
  (3), 26 (1972), pp.~283--290.

\bibitem{HartmanWintner55}
{\sc P.~Hartman and A.~Wintner}, {\em On uniform {D}ini conditions in the
  theory of linear partial differential equations of elliptic type}, Amer. J.
  Math., 77 (1955), pp.~329--354.

\bibitem{Huang96-IUMJ}
{\sc Q.~Huang}, {\em Estimates on the generalized {M}orrey spaces
  {$L^{2,\lambda}_\phi$} and {${\rm BMO}_\psi$} for linear elliptic systems},
  Indiana Univ. Math. J., 45 (1996), pp.~397--439.

\bibitem{JinMvSchaftingen09}
{\sc T.~Jin, V.~Maz'ya, and J.~Van~Schaftingen}, {\em Pathological solutions to
  elliptic problems in divergence form with continuous coefficients}, C. R.
  Math. Acad. Sci. Paris, 347 (2009), pp.~773--778.

\bibitem{Krylov07-CPDE}
{\sc N.~V. Krylov}, {\em Parabolic and elliptic equations with {VMO}
  coefficients}, Comm. Partial Differential Equations, 32 (2007), pp.~453--475.

\bibitem{LiNirenberg03-CPAM}
{\sc Y.~Li and L.~Nirenberg}, {\em Estimates for elliptic systems from
  composite material}, vol.~56, 2003, pp.~892--925.
\newblock Dedicated to the memory of J\"{u}rgen K. Moser.

\bibitem{Li17-CAM}
{\sc Y.~Y. Li}, {\em On the {$C^1$} regularity of solutions to divergence form
  elliptic systems with {D}ini-continuous coefficients}, Chinese Ann. Math.
  Ser. B, 38 (2017), pp.~489--496.

\bibitem{Lieberman87}
{\sc G.~M. Lieberman}, {\em H\"{o}lder continuity of the gradient of solutions
  of uniformly parabolic equations with conormal boundary conditions}, Ann.
  Mat. Pura Appl. (4), 148 (1987), pp.~77--99.

\bibitem{MazyaMcOwen11-JDE}
{\sc V.~Maz'ya and R.~McOwen}, {\em Differentiability of solutions to
  second-order elliptic equations via dynamical systems}, J. Differential
  Equations, 250 (2011), pp.~1137--1168.

\bibitem{Morrey54}
{\sc C.~B. Morrey, Jr.}, {\em Second-order elliptic systems of differential
  equations}, in Contributions to the theory of partial differential equations,
  Annals of Mathematics Studies, no. 33, Princeton University Press, Princeton,
  N.J., 1954, pp.~101--159.

\bibitem{Serrin64}
{\sc J.~Serrin}, {\em Pathological solutions of elliptic differential
  equations}, Ann. Scuola Norm. Sup. Pisa Cl. Sci. (3), 18 (1964),
  pp.~385--387.

\bibitem{Stroffolini01-PA}
{\sc B.~Stroffolini}, {\em Elliptic systems of {PDE} with {BMO}-coefficients},
  Potential Anal., 15 (2001), pp.~285--299.

\end{thebibliography}
\end{document}